\documentclass[11pt, reqno]{amsart}
\usepackage{psfrag}
\usepackage[parfill]{parskip}
\usepackage{amsmath, amsthm, amssymb}
\usepackage{epsfig}
\usepackage{verbatim}
\usepackage{multicol}
\usepackage{url}
\usepackage{latexsym}
\usepackage{mathrsfs}
\usepackage[colorlinks, bookmarks=true]{hyperref}
\usepackage{graphicx}
\usepackage{amsmath}
\usepackage{geometry}

\usepackage[backend=bibtex, style=numeric, maxbibnames=99]{biblatex}
\begin{document}
\newcommand{\mb}{\mathbb}
\newcommand{\mf}{\mathfrak}
\newcommand{\mc}{\mathcal}
\newcommand{\mbf}{\mathbf}
\newcommand{\im}{{\rm im}}
\newcommand{\del}{\partial}
\newcommand{\coker}{{\rm coker}}
\newcommand{\Tr}{\operatorname{Tr}}
\newcommand{\vol}{\operatorname{vol}}
\newcommand{\area}{\operatorname{area}}
\newcommand{\diag}{\text{diag}}

\newcommand{\ZZ}{\mathbb{Z}}
\newcommand{\QQ}{\mathbb{Q}}
\newcommand{\RR}{\mathbb{R}}
\newcommand{\CC}{\mathbb{C}}
\newcommand{\EE}{\mathbb{E}}
\newcommand{\PP}{\mathbb{P}}

\theoremstyle{plain}
  \newtheorem{theorem}{Theorem}[section]
  \newtheorem{proposition}[theorem]{Proposition}
  \newtheorem*{propositionnonumber}{Proposition}
  \newtheorem{lemma}[theorem]{Lemma}
  \newtheorem{corollary}[theorem]{Corollary}
  \newtheorem{conjecture}[theorem]{Conjecture}
\theoremstyle{definition}
  \newtheorem{definition}[theorem]{Definition}
  \newtheorem{example}[theorem]{Example}
  \newtheorem{examples}[theorem]{Examples}
  \newtheorem{question}[theorem]{Question}
  \newtheorem{problem}[theorem]{Problem}
  \newtheorem{assumption}[theorem]{Assumption}
  \newtheorem{remark}[theorem]{Remark}
  
  \title{Mesoscopic Perturbations of Large Random Matrices}
  \author{Jiaoyang Huang}
 \address{Department of  Mathematics,
Harvard University}
\email{jiaoyang@math.harvard.edu}
 \maketitle
\begin{abstract}
We consider the eigenvalues and eigenvectors of small rank perturbations of random  $N\times N$ matrices. We allow the rank of perturbation $M$ to increase with $N$, and the only assumption is $M=o(N)$. The spiked population model, proposed by Johnstone \cite{Jo}, is of this kind, in which all the population eigenvalues are 1's except for a few fixed eigenvalues. Our model is more general since we allow the number of non-unit population eigenvalues to grow with the population size. In both additive and multiplicative perturbation models, we study the nonasymptotic relation between the extreme eigenvalues of the perturbed random matrix and those of the perturbation. As $N$ goes to infinity, we derive the empirical distribution of the extreme eigenvalues of the perturbed random matrix. We also compute the appropriate projection of eigenvectors corresponding to the extreme eigenvalues of the perturbed random matrix. We prove that they are approximate eigenvectors of the perturbation. Our results can be regarded as an extension of the finite rank perturbation case to the full generality up to $M=o(N)$.
\end{abstract}

 \section{Introduction}
The study of sample covariance matrices is fundamental in multivariate statistics. When the size of population is fixed, and the sample size goes to  infinity, the sample covariance matrix provides a good approximation of the population covariance matrix. However, this is not the case when the size of the population is comparable to the sample size. For example, in the null case, the spectrum of the sample covariance matrix is governed by the Marchenko-Pastur law, which is obviously different from the spectrum of the identity matrix. In fact when the population size is large and comparable to the sample size, a basic phenomenon is that the sample eigenvalues are more spread out than the population eigenvalues. However, even in this case, the sample eigenvalues still contain a lot of information about the population eigenvalues, especially the information about the extreme eigenvalues of the population covariance matrix.

A basic idea is that if there are some eigenvalues of the sample covariance matrix well separated from the rest of the sample eigenvalues, one can infer that the population covariance matrix is also ``spiked" with a few significant eigenvalues. Based on this idea, Johnstone proposed the spiked population model \cite{Jo}, in which all the population eigenvalues are one except for a few fixed eigenvalues. Here a key practical question emerges: how many extreme eigenvalues should be retained as being ``significant"? We do not know the number of the extreme population eigenvalues in advance. Also it is possible that the number of the extreme eigenvalues of the population covariance matrix grows with the size of the population. If this is the case, what is the relation between the extreme eigenvalues of the sample covariance matrix and those of the population matrix?

To address these questions, we study the mesoscopic perturbations of large random matrices in the following two models: The additive perturbation model (if $W_N$ is the Wigner matrix, this model is also called the deformed Wigner model), 
\begin{align*}
\tilde{W}_N=W_N+P_N,
\end{align*}
and the multiplicative perturbation model (if $W_N$ is the Wishart matrix, this model is also called the spiked population model),
\begin{align*}
\tilde{W}_N=(I+P_N)^{\frac{1}{2}}W_N(I_N+P_N)^{\frac{1}{2}},
\end{align*}
where $W_N$ is the original random matrix, and $P_N$ is the perturbation. 

Enormous progress has been accomplished in the recent works on studying the influence of these two perturbation models. Most of the settings for the perturbations fall into two categories, either the perturbation is of full rank, or the perturbation is of finite rank.

In the full rank perturbation case, people mainly study the bulk of the spectrum, which refers to the full set of the eigenvalues. When $N$ becomes large, free probability provides us a good understanding of the global behavior of the asymptotic spectrum of the perturbed random matrices. More precisely, under some mild conditions, the empirical eigenvalue distribution of the perturbed random matrices converges to the free convolution of the limit empirical eigenvalue distributions of $W_N$ and  $P_N$ both in expectation and almost surely. We refer to \cite{Vo} and \cite{HP} for pioneering works and \cite{AGZ} for an introduction to free probability. In a series of papers \cite{ABKb,BKa,   BKc, KYb, LS, LSSY, OV}, universality of local eigenvalue statistics was proved for a large class of perturbations in both additive and multiplicative cases.

In the second case when the perturbation is of finite rank, the spectrum is not much altered due to Weyl's interlacement property of eigenvalues. However, the extreme eigenvalues of the perturbed matrix differ from that of the non-perturbed matrix if and only if the eigenvalues of the perturbation are above certain critical threshold. This phenomenon was made precise in \cite{BBP}, where the sharp transition (called the BBP phase transition) was first exhibited for the finite rank multiplicative perturbations of complex Gaussian Wishart matrices. In this case, it was shown that if the eigenvalues of the perturbation are above the threshold, the largest eigenvalue of the perturbed matrix deviates away from the bulk and has Gaussian fluctuation, otherwise it sticks to the bulk and fluctuates according to the Tracy-Widom law. Similar results were proved for the finite rank additive perturbations of complex Gaussian Wigner matrices in \cite{Peb, Pea}. Then in a series of papers \cite{BYa, BYb, BS,BBCF, CDFb, CDFa, CDFF, KYa, PRS, RS}, these results were generalized to Wishart and Wigner random ensembles with non-Gaussian entries. More generally, the results of BBP phase transition were extended to the case when the original matrix $W_N$ and the perturbation $P_N$ are orthogonally (or unitarily) independent in \cite{ BGMa, BGMb,BN}.

In this paper, we study the case when the rank $M$ of the perturbation $P_N$ may increase with $N$, but still in the regime $M=o(N)$. To our knowledge, this kind of perturbations was only studied by P{\'e}ch{\'e} in \cite{Peb, Pea} for deformed complex Gaussian Wigner matrices. P{\'e}ch{\'e} considered the perturbation $P_N = \text{diag} (\pi_1,...,\pi_1,\pi_2,...,\pi_{r_N+1},0,...,0)$ with $\pi_1$ multiplicity $k_N$, such that $k_N=o(N)$ and $r_N=o(N)$ . In that paper, P{\'e}ch{\'e} proved that the largest eigenvalue of the deformed Wigner matrix satisfies Tracy-Widom law after normalization. However, her method heavily depends on the explicit form of the correlation function of the deformed complex Gaussian Wigner matrices, and is hard to extend to non-Gaussian cases. In our paper, we detect the location of the extreme eigenvalues of the perturbed matrix with exponentially high probability in more general settings. Based on it we derive the empirical distribution of the extreme eigenvalues of the perturbed random matrix. Following the paper \cite{BN}, we also investigate the eigenvectors corresponding to the extreme eigenvalues of the perturbed random matrices. 

Our proofs rely on the derivation of the master equation representations, like (\ref{mastereq}), of the eigenvalues and eigenvectors of the perturbed matrix, which is a standard way to study the extreme eigenvalues in the case of finite rank perturbation (such as in \cite{BYa, BYb, BGMa, BN, Tao}). In this case, the master equation reduces the problem of understanding the extreme eigenvalues to the study of a finite rank matrix (the same rank as the perturbation), which is much easier to analyze. One can derive the location and the fluctuation of extreme eigenvalues by passing to the limit. In our setting, we get an $M\times M$ matrix $D(z)$, whose size increases with $N$. Instead of passing to the limit, we directly analyze such a matrix. We find that when $z$ is outside the spectrum of the original matrix, $D(z)$ is monotonic. Then we write $D(z)$ as a sum of two matrices, one is deterministic, and the other is close to a scalar matrix with high probability. The extreme eigenvalues of the perturbed matrix are those $z$ such that the determinant of $D(z)$ vanishes. By comparing $D(z)$ with a certain deterministic matrix, and taking advantage of its monotonicity, we are able to detect the location of the extreme eigenvalues with exponentially high probability.  
  
\section{Additive Perturbation Case}
\subsection{Definition and Notation}{\label{sumdef}}
In this section we study the eigenvalues and eigenvectors of an additively perturbed real symmetric (or Hermitian) random matrix $W_N$ by a deterministic diagonal matrix $P_N$ with small rank, 
\begin{align*}
\tilde{W}_N=W_N+P_N. 
\end{align*}
For the random matrix $W_N$, we consider two cases:
\begin{enumerate}
\item $W_N$ is orthogonally (or unitarily) independent with $P_N$. More precisely, $W_N=U^*H_NU$, where $U$ follows the Haar measure on $N\times N$ orthogonal (or unitary) group, and $H_N$ is an $N\times N$ symmetric (or Hermitian) matrix. We denote the ordered eigenvalues of $H_N$ by $\lambda_{1}(H_N)\geq \lambda_{2}(H_N)\geq \cdots\geq \lambda_{N}(H_N)$. Let $\mu_{N}$ be the empirical eigenvalue distribution of $H_N$, i.e.
$\mu_{N}=\frac{1}{N}\sum_{i=1}^{N}\delta_{\lambda_{i}(H_N)}$. Let $m_N(z)$ be the Stieltjes transform of the empirical eigenvalue distribution of $H_N$,
\begin{align*}
m_{N}(z)=\int_{\RR}\frac{d\mu_N(x)}{z-x}=\frac{1}{N}\sum_{i=1}^{N}\frac{1}{z-\lambda_i(H_N)},
\end{align*}
for $z\in (-\infty, \lambda_N(H_N))\cup (\lambda_1(H_N),\infty)$. Since $m_N$ is monotonically decreasing on $(-\infty, \lambda_N(H_N))$ and $(\lambda_1(H_N),\infty)$ respectively, we can define its inverse map $m_N^{-1}$ on $\RR-\{0\}$.  The inverse Stieltjes transform $m_{N}^{-1}$ maps $(0,\infty)$ to $(\lambda_{1}(H_N),\infty)$, and $(-\infty, 0)$ to $(-\infty, \lambda_{N}(H_N))$.

\item $W_N=\frac{1}{\sqrt{N}}[w_{ij}]_{1\leq i,j\leq N}$ is the Wigner ensemble. $w_{ij}$'s are i.i.d. random variables under the symmetric constraint $W_N=W_N^*$, with mean zero and variance one:
\begin{align*}
\EE[w_{ij}]=0,\quad \EE[w_{ij}^2]=1.
\end{align*} 
Moreover, for technical reason, we assume that $w_{ij}$'s satisfy logarithmic Sobolev inequality with constant $\gamma$. Recall that a random variable $w_{ij}$ satisfies the logarithmic Sobolev inequality with constant $\gamma$ if for any differentiable function $f$,
\begin{align*}
\int f^2(x) \log \frac{f^2(x)}{\int f^2(x) dm(x)} dm(x)\leq 2\gamma \int |f'(x)|^2 dm(x),
\end{align*}
where $dm$ is the law of $w_{ij}$. We refer to \cite[Chapter 4]{GZ} for the concentration of measure properties of measures satisfying logarithmic Sobolev inequality.
\end{enumerate}

The perturbation $P_N$ is an $N\times N$ real symmetric (or complex Hermitian) matrix having rank $M$, where $M=o(N)$. It has $M_1$ positive eigenvalues. We denote its nonzero eigenvalues by $\theta_{1}\geq \theta_2\geq \cdots \geq \theta_{M_1}>0> \theta_{M_1+1}\geq \cdots \geq \theta_{M-1}\geq \theta_{M}$. Let $\tilde{P}_M=\diag\{\theta_1,\theta_2,\cdots, \theta_{M-1}, \theta_M\}$ be the $M\times M$ real diagonal matrix, consisting of the nonzero eigenvalues of $P_N$.

\subsection{Orthogonally (or Unitarily) Independent Case}
In this section we study the eigenvalues and eigenvectors of 
\begin{align*}
\tilde{W}_N=U^*H_NU+P_N.
\end{align*}
Without loss of generality, we can assume that $H_N$ and $P_N$ are diagonal matrices, and $P_N=\diag\{\theta_1,\theta_2,\cdots, \theta_M, 0,0,\cdots, 0\}$. Since we concentrate only on the  extreme eigenvalues of the perturbed matrix $\tilde{W}_N$, in the following we fixed a small universal constant $\delta>0$, and study the eigenvalues of $\tilde{W}_N$ on the spectral domain $(\lambda_N(H_N)-\delta, \lambda_1(H_N)+\delta)^\complement$.

\begin{theorem}{\label{eigfreesum}}
$H_N$ and $P_N$ are defined as in Section \ref{sumdef}, we assume that their norms are bounded by some universal constant $B$, i.e. $\|H_N\|\leq B$, $\|P_N\|\leq B$. We denote the eigenvalues of $\tilde{W}_N=U^*H_NU+P_N$ by $\lambda_{1}(\tilde{W}_N)\geq \lambda_{2}(\tilde{W}_N)\geq \cdots\geq \lambda_{N}(\tilde{W}_N)$. For any $0<\epsilon\leq \delta$ (it may depend on $M$ and $N$), and any $1\leq i\leq M$, if the eigenvalue $\theta_i$ of the perturbation $P_N$ satisfies the following separation condition:
\begin{align}
\label{freesumsep}
\begin{split}
m_N^{-1}(\frac{1}{\theta_i})\geq \lambda_1(H_N)+2\delta,\quad \text{if } \theta_i>0,\\
m_N^{-1}(\frac{1}{\theta_{i}})\leq \lambda_N(H_N)-2\delta,\quad \text{if } \theta_i<0,
\end{split}
\end{align} 
then it creates an extreme eigenvalue $\tilde{\lambda}_i$ of $\tilde{W}_N$, where $\tilde{\lambda}_i=\lambda_i(\tilde{W}_N)$ for $\theta_i>0$, and $\tilde{\lambda}_i=\lambda_{N-M+i}(\tilde{W}_N)$ for $\theta_i<0$, with high probability,  
\begin{align}\label{freesumloc}
\PP\left(m_N^{-1}(\frac{1}{\theta_i})-\epsilon \leq \tilde{\lambda}_i\leq m_N^{-1}(\frac{1}{\theta_i})+\epsilon\right)\geq 1-C e^{-c_1N\epsilon^2+c_2M}
\end{align}
where $c_1$, $c_2$ and $C$ depend only on $\delta$ and $B$. 
\end{theorem}

\begin{proof}
The eigenvalues of $U^*H_NU+P_N$ are the solutions of the equation
\begin{align*}
\det (z-H_N-UP_NU^*)= 0.
\end{align*}
For $z$ on the spectral domain $(\lambda_N(H_N)-\delta, \lambda_1(H_N)+\delta)^\complement$, $z-H_N$ is invertible, we have
\begin{align*}
\det(z-H_N-UP_NU^*)
=\det(z-H_N)\det (I-(z-H_N)^{-1} U P_N U^*)\\
=\det(z-H_N)\det (I-\tilde{U}^{*}(z-H_N)^{-1} \tilde{U} \tilde{P}_M),
\end{align*}
where $\tilde{U}$ is the $N\times M$ matrix, consisting of the first $M$ columns of $U$. Therefore the extreme eigenvalues of $U^*H_NU+P_N$ are solutions of the equation, 
\begin{align}
\label{mastereq}\det (\tilde{P}_M^{-1}-\tilde{U}^{*}(z-H_N)^{-1}\tilde{U})= 0.
\end{align}
In the following, we prove (\ref{freesumloc}) for the extreme eigenvalues on the interval $[\lambda_1(H_N)+\delta, +\infty)$. Denote $D(z)=\tilde{P}_M^{-1}-\tilde{U}^{*}(z-H_N)^{-1}\tilde{U}$, which is a real symmetric (or Hermitian) matrix. We can parametrize its eigenvalues as
\begin{align*}
\tau_1(z)\geq \tau_2(z)\geq \cdots\geq \tau_{M-1}(z)\geq \tau_M(z),\quad z\in [\lambda_1(H_N)+\delta, +\infty).
\end{align*}
Notice that for $z>z'\geq\lambda_1(H_N)+\delta$, we have
\begin{align*}
D(z)-D(z')=&\tilde{U}^*\left( (z'-H_N)^{-1}-(z-H_N)^{-1}\right)\tilde{U}\\
               =&(z-z')\tilde{U}^*(z'-H_N)^{-1}(z-H_N)^{-1}\tilde{U}.
\end{align*}
Therefore $D(z)-D(z')$ is positive definite. The following Lemma implies that, for $i=1,2\cdots, M$, $\tau_{i}(z)$'s are all increasing functions of $z$ on the interval $[\lambda_1(H_N)+\delta,+\infty)$.
\begin{lemma}{\label{comparison}}
Let $A$ and $B$ be two $N\times N$ real symmetric (or Hermitian) matrices, such that $A-B$ is positive definite. Denote the $n$-th maximal eigenvalues of $A$ and $B$ as $\lambda_{n}(A)$ and $\lambda_n(B)$ respectively, then we have
\begin{align*}
\lambda_{n}(A)> \lambda_n(B),\quad n=1,2,3,\cdots, N.
\end{align*}
\end{lemma}
\begin{proof}
Take the space $W$ spanned by the eigenvectors of $B$ corresponding to its first $n$ eigenvalues.
By the minimax principle,
\begin{align*}
\lambda_{n}(A)=\max_{\dim V=n}\min_{v\in V} v^*Av\geq \min_{v\in W}v^*Av.
\end{align*}
Denote $A|_W$ and $B|_W$ the restriction of $A$ and $B$ on the space $W$, then we still have that $A|_W-B|_W$ is positive definite. Therefore
\begin{align*}
\min_{v\in W}v^*Av=\min_{v\in W}v^*A|_Wv> \min_{v\in W}v^*B|_Wv=\lambda_n(B).
\end{align*}
\end{proof}
\begin{remark}
If we only have that $A-B$ is positive semi-definite, then one can still show that $\lambda_n(A)\geq \lambda_n(B)$.
\end{remark}
The determinant of $D(z)$ is the product of its eigenvalues,
\begin{align*}
\det(\tilde{P}_M^{-1}-\tilde{U}^*(z-H_N)^{-1}\tilde{U})=\prod_{i=1}^{M}\tau_i(z).
\end{align*}
The determinant vanishes, if and only if some eigenvalue vanishes. Therefore $z\in [\lambda_1(H_N)+\delta,+\infty)$ corresponds to an extreme eigenvalue of $U^*H_NU+P_N$ if and only if $\tau_i(z)=0$ for some $i$. Consider the point process $\{\tau_1(z),\tau_2(z), \cdots, \tau_M(z)\}$ as $z$ goes from $+\infty$ down to $\lambda_1(H_N)+\delta$. For $z$ sufficiently large, such that $z-\lambda_{1}(H_N)>{\theta_{1}}$,
we have $0<\tilde{U}^*(z-H_N)^{-1}\tilde{U}<\theta_1^{-1}$, which is smaller than the smallest positive eigenvalue of $\tilde{P}^{-1}_N$. Therefore  for  $z>\lambda_1(H_N)+\theta_1$, $\tilde{P}_M^{-1}-\tilde{U}^*(z-H_N)^{-1}\tilde{U}$ has $M_1$ positive eigenvalues: $\tau_1(z)\geq\tau_2(z)\geq\cdots\geq \tau_{M_1}(z)>0>\tau_{M_1+1}(z)\geq \cdots \geq \tau_{M}(z).$
As $z$ goes from $+\infty$ down to $\lambda_1(H_N)+\delta$, all the eigenvalues $\tau_i(z)$'s will decrease. At some point $z=z_1$, $\tau_{M_1}(z_1)$ becomes zero, this $z_1$ is exactly the largest extreme eigenvalue $\lambda_1(\tilde{W}_N)$. Later at some point $z=z_2$, $\tau_{M_1-1}(z_2)$ becomes zero, this $z_2$ is the second largest extreme eigenvalue $\lambda_2(\tilde{W}_N)$. We may now continue in this manner, sequentially get all the extreme eigenvalues of $\tilde{W}_N$ on the interval $[\lambda_1(H_N)+\delta, +\infty)$.
 
More formally, we introduce the counting function:
\begin{align*}
\mf{n}(z)=\#\{i\in \{1,2,\cdots, M\}: \tau_i(z)\geq 0\}.
\end{align*}
Since $\tau_i(z)$'s are all strictly increasing function of $z$ on the interval $[\lambda_1(H_N)+\delta, +\infty)$, $\mf{n}(z)$ is a right continuous, non-decreasing function. The $i$-th largest extreme eigenvalue of $\tilde{W}_N$ is given by
\begin{align*}
\lambda_i(\tilde{W}_N)=\inf_{z\in [\lambda_1+\delta, +\infty) }\{z: \mf{n}(z)\geq M_1-i+1\},
\end{align*} 
if the set on the lefthand side is not empty. Therefore the counting function $\mf{n}(z)$ can be used to detect the location of the $i$-th largest eigenvalue of $\tilde{W}_N$. More precisely, if we can find numbers $L$ and $R$, such that $\lambda_{1}(H_N)+\delta\leq L<R$ and
\begin{align}{\label{criterion}}
\mf{n}(L)\leq M_1-i,\quad \mf{n}(R)\geq M_1-i+1,
\end{align}
then we can conclude from the above argument that $\lambda_i(\tilde{W}_N)\in (L, R]$.

In the following, we prove (\ref{freesumloc}). For any $\epsilon>0$, we take
\begin{align*}
L=m_N^{-1}(\frac{1}{\theta_i})-\epsilon,\quad R=m_N^{-1}(\frac{1}{\theta_i})+\epsilon.
\end{align*}
in  (\ref{criterion}). Then we use Proposition \ref{orthogonalcase}, which states that roughly $\tilde{U}^{*}(z-H_N)^{-1}\tilde{U}\approx m_N(z)I_M$. If we take $A=(L-H_N)^{-1}$ and $\xi=\frac{\epsilon}{8B^2}$ (recall that the norm of $H_N$ and $P_N$ are bounded by the constant $B$) in Proposition \ref{orthogonalcase}, we get
\begin{align}
\notag\tilde{P}_M^{-1}-\tilde{U}^{*}(L-H_N)^{-1}\tilde{U}
\leq& \tilde{P}_M^{-1}-(m_N(L)-\frac{\epsilon}{8B^2})I_M\\
\label{needbound}\leq& \tilde{P}_M^{-1}-(\frac{1}{\theta_i}+\frac{\epsilon}{4B^2}-\frac{\epsilon}{8B^2})I_M\\
\notag=&  \tilde{P}_M^{-1}-(\frac{1}{\theta_i}+\frac{\epsilon}{8B^2})I_M,
\end{align}
with probability at least $1-Ce^{-c_1N\epsilon^2+c_2M}$, where $C$, $c_1$ and $c_2$ depend only on $B$ and $\delta$ (as in \eqref{freesumsep}). One can find the proof for the inequality (\ref{needbound}) in the Appendix Proposition \ref{inequalitySandT}. Since $\tilde{P}_M^{-1}-(\frac{1}{\theta_i}+\frac{\epsilon}{8B^2})I_M$ has at most $M_1-i$ non-negative eigenvalues, we know that $\mf{n}(L)\leq M_1-i$ with high probability. Similarly, if we take $A=(R-H_N)^{-1}$ and $\xi=\frac{\epsilon}{8B^2}$ in Proposition \ref{orthogonalcase}, by the same argument, we get that $\mf{n}(R)\geq M_1-i+1$ with probability at least $1-Ce^{-c_1N\epsilon^2+c_2M}$. Therefore we can conclude that the  $i$-th largest eigenvalue of $\tilde{W}_N$ is between $L$ and $R$ with exponentially high probability. This finishes the proof of (\ref{freesumloc}).
\end{proof}
\begin{remark}
Since in our setting, the rank of perturbation $P_N$ is much smaller than $N$, $M=o(N)$, (\ref{freesumloc}) in fact gives us the location of extreme eigenvalues of the perturbed matrix $\tilde{W}_N$ with exponentially high probability. Moreover, (\ref{freesumloc}) implies that the fluctuations of the extreme eigenvalues is at most of order $\sqrt{\frac{M}{N}}$.
\end{remark}

\begin{remark}
If we take $M$ as a fixed number, then Theorem \ref{eigfreesum} uncovers the well-known result of the finite rank perturbation of random matrices. Moreover, in this case, Theorem \ref{eigfreesum} gives us the correct order of the fluctuation of the extreme eigenvalues. In fact it is proved in \cite{BGMa} that the extreme eigenvalues exhibit the Gaussian fluctuation with order $N^{-\frac{1}{2}}$, and the joint distribution of normalized extreme eigenvalues converges to the law of the eigenvalues of independent GUE or GOE random matrices. However we do not know how to generalize their method to give the joint law of the extreme eigenvalues in mesoscopic case. 
\end{remark}

\begin{remark}
If there exists a large gap between two adjacent eigenvalues of $H_N$, i.e. $\lambda_j(H_N)-\lambda_{j+1}(H_N)\gg \delta$ for some $j$, Theorem \ref{eigfreesum} can be adapted to describe the extreme eigenvalues on the interval $(\lambda_{j+1}(H_N), \lambda_j(H_N))$. In fact the Stieltjes transform $m_N(z)$ of the empirical eigenvalue distribution of $H_N$ is well defined on $(\lambda_{j+1}(H_N), \lambda_j(H_N))$, and is strictly decreasing. Therefore for any eigenvalue $\theta_i$ of the perturbation $P_N$, under certain separation condition analogue to (\ref{freesumsep}), it will create an extreme eigenvalue around $m_N(\frac{1}{\theta_i})\in (\lambda_{j+1}(H_N), \lambda_j(H_N) )$ with high probability. This comment applies to the other theorems in this paper.
\end{remark}

\begin{theorem}{\label{evfreesum}}
$H_N$ and $P_N$ are defined as in Section \ref{sumdef}, we assume that their norms are bounded by some universal constant $B$, i.e. $\|H_N\|\leq B$, $\|P_N\|\leq B$. We denote the eigenvalues of $\tilde{W}_N=U^*H_NU+P_N$ by $\lambda_{1}(\tilde{W}_N)\geq \lambda_{2}(\tilde{W}_N)\geq \cdots\geq \lambda_{N}(\tilde{W}_N)$. For any $1\leq i\leq M$, such that the eigenvalue $\theta_i$ of the perturbation $P_N$ satisfies the separation condition (\ref{freesumsep}), with high probability, it will create an extreme eigenvalue $\tilde{\lambda}_i$, where $\tilde{\lambda}_i=\lambda_i(\tilde{W}_N)$ for $\theta_i>0$, and $\tilde{\lambda}_i=\lambda_{N-M+i}(\tilde{W}_N)$ for $\theta_i<0$. We denote the corresponding eigenvector by $v_i$ (if $\tilde{\lambda}_i$ is not a simple eigenvalue, $v_i$ can be any of its eigenvectors).  We denote the projection of $v_i$ on the eigenspace of  the nonzero eigenvalues of $P_N$ by $\tilde{v}_i$, which is the projection of $v_i$ on the first $M$ coordinates. For any $0\leq \epsilon\leq\delta$, which may depend on $M$ and $N$, with probability at least $1-C e^{-c_1N\epsilon^2+c_2M}$, where $c_1$, $c_2$ and $C$ depend only on $\delta$ and $B$, we have
\begin{enumerate}
\item The norm of $\tilde{v}_i$ satisfies
\begin{align}
\label{proj}\left|\left\|\tilde{v}_i\right\|^2+\frac{1}{\theta_i^2m_N'(m_N^{-1}(\frac{1}{\theta_i}))}\right|\leq \epsilon C_{B,\delta}.
\end{align}

\item The projection $\tilde{v}_i$ is an approximate eigenvector of $\tilde{P}_M$ with eigenvalue $\theta_i$,
\begin{align}
\label{loc}\left\|\tilde{P}_M \tilde{v}_i-\theta_i\tilde{v}_i \right\|\leq \epsilon C_{B,\delta},
\end{align}
\end{enumerate}
where $C_{B,\delta}$ depends on $\delta$ and $B$.
\end{theorem}

\begin{proof}
Since $v_i$ is a eigenvector of $U^*H_NU+\tilde{P}_M$ corresponding to the eigenvalue $\tilde{\lambda}_i$,
\begin{align*}
\tilde{\lambda}_iv_i=(U^*H_NU+P_N)v_i.
\end{align*}
This leads to 
\begin{align}
\label{selfev}v_i=U^*(\tilde{\lambda}_i-H_N)^{-1}UP_Nv_i.
\end{align}
We can obtain the following two equations for $\tilde{v}_i$ from (\ref{selfev}). (\ref{ch1}) is from projecting both sides of (\ref{selfev}) on the first $M$ coordinates; and (\ref{ch2}) is from taking the norm square on both sides of (\ref{selfev}).
\begin{align}
\label{ch1}\left(I_M-\tilde{U}^*(\tilde{\lambda}_i-H_N)^{-1}\tilde{U}\tilde{P}_M\right)\tilde{v}_i=0,\\
\label{ch2}\tilde{v}_i^*\tilde{P}_M \left(\tilde{U}^*(\tilde{\lambda}_i-H_N)^{-2}\tilde{U}\right)\tilde{P}_M\tilde{v}_i=1,
\end{align}
where $\tilde{U}$ is the $N\times M$ matrix, consisting of the first $M$ columns of $U$. Approximately, we have $\tilde{U}^*(\tilde{\lambda}_i-H_N)^{-1}\tilde{U}\approx \frac{1}{\theta_i}I_M$ and $\tilde{U}^*(\tilde{\lambda}_i-H_N)^{-2}\tilde{U}\approx -m_N'(m_N^{-1}(\frac{1}{\theta_i}))I_M$, which will simplify the algebraic relations (\ref{ch1}) and (\ref{ch2}). In the following, we will make these approximations more quantitive. We have the following inequalities, by taking $K=\max\{\frac{8}{\delta^2}, \frac{16}{\delta^3}\}$, $L=m_N^{-1}(\frac{1}{\theta_i})-\frac{\epsilon}{K}$, and $R=m_N^{-1}(\frac{1}{\theta_i})+\frac{\epsilon}{K}$ in Proposition \ref{inequalitySandT}, 
\begin{align*}
\frac{1}{\theta_i}-\frac{\epsilon}{2}\leq m_N(R)\leq \frac{1}{\theta_i}-\frac{\epsilon}{4KB^2}
< \frac{1}{\theta_i}+\frac{\epsilon}{4KB^2}\leq m_N(L)\leq \frac{1}{\theta_i}+\frac{\epsilon}{2},\\
-m_N'(m_N^{-1}(\frac{1}{\theta_i}))-\frac{\epsilon}{2}\leq -m_N'(R)\leq -m_N'(L)\leq
-m_N'(m_N^{-1}(\frac{1}{\theta_i}))+\frac{\epsilon}{2}.
\end{align*}
By taking $\xi=\min\{\frac{\epsilon}{2}, \frac{\epsilon}{8KB^2}\}$, and $A=(L-H_N)^{-1}$, $(R-H_N)^{-1}$, $(L-H_N)^{-2}$ and $(R-H_N)^{-2}$ respectively in Proposition \ref{orthogonalcase}, we get 
\begin{align}
\label{eiglocation}(\frac{1}{\theta_i}-\epsilon)I_M\leq \tilde{U}^{*}(R-H_N)^{-1}\tilde{U}< \frac{1}{\theta_i}I_M< \tilde{U}^{*}(L-H_N)^{-1}\tilde{U}\leq (\frac{1}{\theta_i}+\epsilon)I_M,
\end{align}
and
\begin{align}
\notag -(m_N'(m_N^{-1}(\frac{1}{\theta_i}))+\epsilon)I_M&\leq \tilde{U}^{*}(R-H_N)^{-2}\tilde{U}\\
\label{normbound}&\leq \tilde{U}^{*}(L-H_N)^{-2}\tilde{U}\leq-(m_N'(m_N^{-1}(\frac{1}{\theta_i}))-\epsilon)I_M,
\end{align}
with exponentially high probability. We denote the event such that (\ref{eiglocation}) and (\ref{normbound}) hold by $\mc{A}$. In the following we show that (\ref{proj}) and (\ref{loc}) hold on $\mc{A}$. 
The same argument as in the proof of Theorem \ref{eigfreesum}, (\ref{eiglocation}) implies that $\tilde{\lambda}_i\in [L,R]$. Since both $\tilde{U}^{*}(z-H_N)^{-1}\tilde{U}$ and $\tilde{U}^{*}(z-H_N)^{-2}\tilde{U}$ are monotonic as a function of $z$ outside the spectrum of $H_N$. Especially, they are monotonic on $[L,R]$, (\ref{eiglocation}) and (\ref{normbound}) imply
\begin{align}
\label{approx1}(\frac{1}{\theta_i}-\epsilon)I_M\leq& \tilde{U}^{*}(\tilde{\lambda}_i-H_N)^{-1}\tilde{U}\leq (\frac{1}{\theta_i}+\epsilon)I_M,\\
\label{approx2}-(m_N'(m_N^{-1}(\frac{1}{\theta_i}))+\epsilon)I_M\leq& \tilde{U}^{*}(\tilde{\lambda}_i-H_N)^{-2}\tilde{U}\leq
-(m_N'(m_N^{-1}(\frac{1}{\theta_i}))-\epsilon)I_M,
\end{align}
on the event $\mc{A}$. With the estimate (\ref{approx1}), (\ref{ch1}) can be reduced to
\begin{align}
\label{bound1}\left\|\left(I_M-\frac{1}{\theta_i}\tilde{P}_M\right)\tilde{v}_i\right\|\leq \epsilon B,
\end{align}
which gives us (\ref{loc}). For (\ref{ch2}), using the approximation $\tilde{U}^{*}(\tilde{\lambda}_i-H_N)^{-2}\tilde{U}\approx-m_N'(m_N^{-1}(\frac{1}{\theta_i}))I_M$ from (\ref{approx2}), and $\frac{1}{\theta_i}\tilde{P}_M\tilde{v}_i\approx \tilde{v}_i$ from (\ref{bound1}), we get
\begin{align}
\left\|\tilde{v}_i\right\|^2=\left\|\frac{1}{\theta_i}\tilde{P}_M\tilde{v}_i\right\|^2+O_{B,\delta}(\epsilon)=\frac{1}{-\theta_i^2m_N'(m_N^{-1}(\frac{1}{\theta_i}))}+O_{B,\delta}(\epsilon).
\end{align}
This finishes the proof of (\ref{proj}). 
\end{proof}

\subsection{Wigner Case}
In this section we study the eigenvalues of deformed Wigner matrices,
\begin{align*}
\tilde{W}_N=W_N+P_N.
\end{align*}
Since in this case $W_N$ and $P_N$ are no longer orthogonally invariant, we can not assume that $P_N$ is diagonal. However since $W_N$ is permutation invariant, we can still assume that $P_N= V \tilde{P}_M V^*$, where $V$ is an $N\times M$ matrix, consisting of the eigenvectors corresponding to the nonzero eigenvalues of $P_N$. 

It is well known that the empirical eigenvalue distribution of $W_N$ converges to the semi-circle law with density
\begin{align*}
\rho_{sc}(x)=\frac{\sqrt{4-x^2}}{2\pi},\quad \text{for } x\in [-2,2].
\end{align*}
We denote its Stieltjes transform as
\begin{align*}
m_{sc}(z)=\frac{z-\text{sgn}(z)\sqrt{z^2-4}}{2}, \quad \text{for } z\in (-\infty, -2]\cup [2,+\infty).
\end{align*}
In the following we fix a small universal constant $\delta>0$. It is known that with high probability the eigenvalues of Wigner matrix $W_N$ are in $(-2-\frac{\delta}{2}, 2+\frac{\delta}{2})$. Since we concentrate only on the  extreme eigenvalues of the perturbed matrix $\tilde{W}_N$,  we study the eigenvalues of $\tilde{W}_N$ on the spectral domain $(-2-\delta, 2+\delta)^\complement$.

\begin{theorem}{\label{wigsum}}
Wigner matrix $W_N$ and perturbation $P_N$ are defined as in Section \ref{sumdef}, we assume the norm of $P_N$ is bounded by some universal constant $B$, i.e. $\|P_N\|\leq B$. We denote the eigenvalues of $\tilde{W}_N=W_N+P_N$ by $\lambda_{1}(\tilde{W}_N)\geq \lambda_{2}(\tilde{W}_N)\geq \cdots\geq \lambda_{N}(\tilde{W}_N)$. For any $\sqrt{\frac{M}{N}}\ll\epsilon\leq \delta$ (it may depend on $M$ and $N$), and any $1\leq i\leq M$, if $\theta_i$ satisfies the following separation condition
\begin{align}
\label{wigsep}|\theta_i|\geq 1+2\delta,
\end{align} 
then it will create an extreme eigenvalue $\tilde{\lambda}_i$ of $\tilde{W}_N$, where $\tilde{\lambda}_i=\lambda_i(\tilde{W}_N)$ for $\theta_i>0$, and $\tilde{\lambda}_i=\lambda_{N-M+i}(\tilde{W}_N)$ for $\theta_i<0$, with high probability,  
\begin{align}\label{wigloc}
\PP\left(\theta_i+\frac{1}{\theta_i}-\epsilon \leq \tilde{\lambda}_i\leq \theta_i+\frac{1}{\theta_i}+\epsilon\right)\geq 1-C e^{-c_1N\epsilon^2+c_2M}
\end{align}
where $c_1$, $c_2$ and $C$ depend only on $\delta$, $B$ and the logarithmic Sobolev constant $\gamma$. 
\end{theorem}

\begin{proof}
The proof is similar to the proof of the orthogonally independent case. In fact, if we condition on $\|W_N\|\leq 2+\frac{\delta}{2}$, then for any $z\in (-2-\delta, 2+\delta)^\complement$, we have 
\begin{align*}
\det(z- W_N-P_N)=\det(z-W_N)\det(\tilde{P}_M)\det(\tilde{P}_M^{-1}-V^* (z-W_N)^{-1}V)
\end{align*}
We can go through exactly the same argument as in Theorem \ref{eigfreesum} if we have the following proposition, which will take place the role of Proposition \ref{orthogonalcase} in the proof of orthogonally independent case.
\end{proof}
\begin{proposition}{\label{wignerconcen}}
$W_N=\frac{1}{\sqrt{N}}[w_{ij}]_{1\leq i,j\leq N}$ is the Wigner ensemble as in Section \ref{sumdef}. We denote the set of symmetric matrices $D=\{H\in \mbf{M}_{N\times N}(\RR): \|H\|\leq 2+\frac{\delta}{2}, H=H^*\}$. Then for any $z\in (-2-\delta, 2+\delta)^{\complement}$, $\xi>0$, and $N\times M$ matrix $V$ with orthonormal columns, we have
\begin{align}\label{iidconcen}
\PP\left(W_N\in D \text{ and }\left\|V^*(z-W_N)^{-1}V- m_{sc}(z)I_M\right\|\leq \xi\right)\geq 1- Ce^{-cN\xi^2+M\ln 7} 
\end{align}
where $c$ and $C$ depend on $\delta$ and the logarithmic Sobolev constant $\gamma$.
\end{proposition}
\begin{proof}
Under our assumption, it is well know that $\EE[\|W_N\|]=2$. The norm of a symmetric matrix is a Lipschitz continuous function with Lipschitz constant $2$, i.e. $\|X\|-\|Y\|\leq \|X-Y\|_{HS}$ for any two symmetric matrices $X$ and $Y$. By logarithmic Sobolev inequality we know $\PP(|\|W_N\|-2|\geq \frac{\delta}{2})\leq 1-e^{cN\delta^2}$, where the constant $c$ depends on the logarithmic Sobolev constant $\gamma$. Therefore $W_N\in D$ with exponentially high probability. 

Next we show that for any unit vector $u\in S^{N-1}$,
\begin{align}\label{singlevector}
\PP\left(W_N \in D \text{ and } \left |u^*(z-W_N)^{-1}u- m_{sc}(z)\right |\leq \frac{\xi}{3}\right)\geq 1-Ce^{-cN\xi^2} .
\end{align}
For any $z\in (2-\delta, 2+\delta)^{\complement}$, we define the function $\tilde{f}(X)= u^*(z-X)^{-1}u$ for $X\in D$. Notice that for any $X, Y\in D$, we have
\begin{align*}
|\tilde{f}(X)-\tilde{f}(Y)|=|u^* (z-X)^{-1}(Y-X)(z-Y)^{-1}u|\leq \frac{4}{\delta^2} \|Y-X\|_{HS}.
\end{align*}
Thus $\tilde{f}$ is Lipschitz continuous on $D$. We can extend $\tilde{f}$ to the whole space of symmetric matrices as
\begin{align*}
f(X)=\sup_{Y\in D}\left\{ f(Y)- \frac{4}{\delta^2}\|Y-X\|_{HS}\right\}.
\end{align*}
$f$ agrees with $\tilde{f}$ on $D$ and is a Lipschitz continuous function on the whole space of symmetric matrices. By logarithmic Sobolev inequality we have
\begin{align}{\label{logsob}}
\PP(|f(W_N)-\EE[f(W_N)]|\geq \frac{\xi}{6})\leq e^{-cN \xi^2},
\end{align}
where $c$ depends on $\delta$ and the logarithmic Sobolev constant $\gamma$. Next we show that for $N$ large enough,  $|\EE[f(W_N)]-m_{sc}(z)|\leq \frac{\xi}{6}$. In \cite[Theorem 2.3]{KYc}, Knowles and Yin proved that  for Wigner ensemble $W_N$ such that the entries have uniformly subexponential decay (which is definitely true in our case), for any deterministic vectors $v,w\in S^{n-1}$, the following holds 
\begin{align}
\label{isolaw}|\langle v, (z+i\eta-W_N)^{-1}, w\rangle-m_{sc}(z+i\eta)\langle v, w\rangle|\leq N^{\epsilon}\sqrt{\frac{\text{Im } m_{sc}(z+i\eta)}{N\eta}},
\end{align}
with overwhelming probability for any $\epsilon>0$ and $\eta>0$. Notice that away from the asymptotic spectrum $[-2,2]$, both $\langle v, (z+i\eta-W_N)^{-1}w\rangle$ and $m_{sc}(z+i\eta)$ are Lipschitz continuous with Lipschitz constant $O(\delta^{-2})$. Therefore we are able to extend \eqref{isolaw} to the real axis in our setting, i.e. $z\in [-2-\delta, 2+\delta]^\complement$.
\begin{align*}
&|\EE[f(W_N)]-m_{sc}(z)|\\
\leq&| \EE[\mbf{1}_{D^{\complement}} (f(W_N)-u^*(z+i\eta-W_N)^{-1}u)]|+
|\EE[\mbf{1}_{D} (u^*(z-W_N)^{-1}u-u^*(z+i\eta-W_N)^{-1}u)]|\\
+&|m_{sc}(z+i\eta)-m_{sc}(z))|+|\EE[u^*(z+i\eta -W_N)^{-1}u]-m_{sc}(z+i\eta)|\\
\leq& (c_1+\frac{1}{\eta})e^{-cN\delta^2}+c_2\eta+c_3\eta+ \frac{c_4}{N^{\frac{1}{2}-\epsilon}\eta^{\frac{1}{2}}}
\end{align*}
where $c_1,c_2,c_3,c_4$ are universal constants depending only on $\delta$. Therefore we can take $\eta=\frac{1}{\sqrt{N}}$ and for $N$ large enough we will have $|\EE[f(W_N)]-m_{sc}(z)|\leq \frac{\xi}{6}$.  If we combine this and (\ref{logsob}), we will have
\begin{align*}
&\PP(W_N\in D\text{ and } |u^*(z-W_N)^{-1}u -m_{sc}(z)|\leq \frac{\xi}{3})\\
=&\PP(W_N\in D\text{ and } |f(W_N) -m_{sc}(z)|\leq \frac{\xi}{3})\\
\geq& \PP(|f(W_N) -m_{sc}(z)|\leq \frac{\xi}{3})-\PP(W_N\notin D)
\geq 1-Ce^{-cN\xi^2},
\end{align*}
where $c$ depends on $\delta$ and the logarithmic Sobolev constant $\gamma$. By the same epsilon-net argument as in Proposition \ref{orthogonalcase}, we can derive (\ref{iidconcen}) from \eqref{singlevector}.
\end{proof}

With the rigidity result of the extreme eigenvalues from Theorem \ref{wigsum}, one can easily derive the empirical distribution of these extreme eigenvalues if the empirical distribution of nonzero eigenvalues of $P_N$ goes to some limit. In fact we have the following corollary,  

\begin{corollary}
Wigner ensemble $W_N$ and perturbation $P_N$ are defined as in Section \ref{sumdef}, we assume that the norm of $P_N$ is bounded by some universal constant $B$, i.e. $\|P_N\|\leq B$. Moreover we assume that the empirical distribution of the nonzero eigenvalues of $P_N$ converges weakly to a compactly supported measure $\nu$,
\begin{align*}
\hat{\nu}_{N}=\frac{1}{M}\sum_{i=1}^{M}\delta_{\theta_i}\rightarrow \nu,\quad N\rightarrow +\infty,
\end{align*} 
with support supp $\nu \subset [a, b]$, and $a>1$. We denote the eigenvalues of $\tilde{W}_N=W_N+P_N$ as $\lambda_1(\tilde{W}_N)\geq \lambda_2(\tilde{W}_N)\geq \cdots \geq \lambda_N(\tilde{W}_N)$. Then almost surely, the empirical distribution of the largest $M$ eigenvalues of $\tilde{W}_N$ converges weakly to the push forward measure $\gamma_\# \nu$, 
\begin{align*}
\frac{1}{M}\sum_{i=1}^{M}\sigma_{\lambda_i(\tilde{W}_N)}\rightarrow \gamma_\#\nu, \quad a.s.\end{align*}
as $N$ goes to $+\infty$, where $\gamma(\theta)=\theta+\frac{1}{\theta}$.
\end{corollary}
\begin{proof}
We fix a constant $\delta=\frac{1}{3}(a-1)$, and define the event 
\begin{align*}
A_N=\bigcap_{i: \theta_i\geq a-\delta}\left\{\left|\lambda_i(\tilde{W}_N)-\theta_i-\frac{1}{\theta_i}\right|\leq \epsilon_N\right\},
\end{align*}
where $\epsilon_N$ will be chosen later. Since for these $\theta_i\geq a-\delta$, they satisfy the separation condition (\ref{wigsep}): $\theta_i\geq a-\delta\geq 1+2\delta$. From theorem \ref{wigsum}, there exist constants $c_1$, $c_2$, $c_3$ and $C$ such that $\PP\left(A_N^\complement\right)\leq  CM e^{-c_1N\epsilon_N^2+c_2M}$.
Therefore if we take $\epsilon_N=\min\{\delta, \left(\frac{2c_2}{c_1}\right)^{\frac{1}{2}} \left(\frac{M}{N}\right)^{\frac{1}{4}}\}$, since $M=o(N)$, we have $\epsilon_N\rightarrow 0$. Moreover by Borel-Cantelli lemma, almost surely $A_N$'s hold. Due to the weak convergence of $\hat{\nu}_N$, we need to show that, for any bounded Lipschitz test function $f$, that
\begin{align*}
\lim_{N\rightarrow \infty} \left(\frac{1}{M}\sum_{i=1}^{M}\big(f(\lambda_i(\tilde{W}_N))-f(\theta_i+\frac{1}{\theta_i})\big)\right)=0.
\end{align*}
This term can be decomposed into two parts. The first part corresponds to $\theta_i$'s outside of the support of $\nu$. Due to weak convergence of $\hat{\nu}_N$, the portion of such $\theta_i$'s goes to zero,
\begin{align*}
\left|\frac{1}{M}\sum_{i: \theta_i\leq a-\delta}\big(f(\lambda_i(\tilde{W}_N))-f(\theta_i+\frac{1}{\theta_i})\big)\right|\leq \frac{2\|f\|_{\infty}}{M}\#\{i: \theta_i\leq a-\delta\}\rightarrow 0.
\end{align*}
The second part corresponds to $\theta_i$'s, such that $\theta_i>a-\delta$. Since almost surely $A_N$'s hold, which implies that simultaneously $\left|\lambda_i(\tilde{W}_N)-\theta_i+\frac{1}{\theta_i}\right|\leq \epsilon_N$. We can control the second term by the Lipschitz norm of $f$, 
\begin{align*}
\left|\frac{1}{M}\sum_{i: \theta_i> a-\delta}\big(f(\lambda_i(\tilde{W}_N))-f(\theta_i+\frac{1}{\theta_i})\big)\right|
\leq\epsilon_N\|f\|_{\mc{L}}\rightarrow 0,\quad a.s.
\end{align*}
Therefore almost surely we have the convergence of the empirical distribution of the largest $M$ eigenvalues of $\tilde{W}_N$,
\begin{align*}
\lim_{N\rightarrow \infty}\frac{1}{M}\sum_{i=1}^{M}f(\lambda_i(\tilde{W}_N))
=\lim_{N\rightarrow \infty}\sum_{i=1}^{M} f(\theta_i+\frac{1}{\theta_i})
=\int f(\theta) d\gamma_\# \nu.
\end{align*}
\end{proof}
\begin{remark}
If we assume that the extreme nonzero eigenvalues of $P_N$ converge to the edge of the limit distribution $\nu$, then almost surely we have 
\begin{align*}
\lambda_1(\tilde{W}_N)\rightarrow b+\frac{1}{b},\quad \lambda_M(\tilde{W}_N)\rightarrow a+\frac{1}{a}.
\end{align*}
\end{remark}

\section{Multiplicative Perturbation Case}
\subsection{Definition and Notation}{\label{multdef}}
In this section we study the eigenvalues and eigenvectors of a multiplicative perturbed real symmetric (or Hermitian) random matrix $W_N$ by a deterministic diagonal matrix $P_N$ with small rank, 
\begin{align*}
\tilde{W}_N=(I+P_N)^{\frac{1}{2}}W_N(I_N+P_N)^{\frac{1}{2}}. 
\end{align*}
For the random matrix $W_N$, we consider two cases:
\begin{enumerate}
\item $W_N$ is orthogonally (or unitarily) independent with $P_N$. More precisely, $W_N=U^*H_NU$, where $U$ follows the Haar measure on $N\times N$ orthogonal (or unitary) group, and $H_N$ is an $N\times N$ nonzero positive semi-definite symmetric (or Hermitian) matrix. We denote the ordered eigenvalues of $H_N$ by $\lambda_{1}(H_N)\geq \lambda_{2}(H_N)\geq \cdots\geq \lambda_{N}(H_N)\geq 0$. Let $\mu_{N}$ be the empirical eigenvalue distribution of $H_N$, i.e.
$\mu_{N}=\frac{1}{N}\sum_{i=1}^{N}\delta_{\lambda_{i}(H_N)}$. Let $T_N(z)$ be the T-transform of the empirical eigenvalue distribution of $H_N$,
\begin{align*}
T_{N}(z)=\int_{\RR}\frac{x}{z-x}d\mu_N(x)=\frac{1}{N}\sum_{i=1}^{N}\frac{\lambda_i(H_N)}{z-\lambda_i(H_N)},
\end{align*}
for $z\in (-\infty, \lambda_N(H_N))\cup (\lambda_1(H_N),\infty)$. Similar to the Stieltjes transform, $T_N^{-1}$ is well defined on $\RR-\{0\}$. It maps $(0, \infty)$ to $(\lambda_{1}(H_N), \infty)$, and $(-\infty, 0)$ to $(-\infty, \lambda_{N}(H_N))$.
\item $W_N=X_NX_N^*$ is a Wishart ensemble, where $X_N=\frac{1}{\sqrt{p}}[x_{ij}]_{1\leq i\leq N\atop 1\leq j \leq p}$ is an $N\times p$ random matrix. As $N\rightarrow +\infty$, the ratio $\frac{N}{p}$ goes to some positive constant less than one, i.e. $\frac{N}{p}\rightarrow \phi\in (0,1)$.  $x_{ij}$'s are i.i.d. random variables with mean zero and variance one:
\begin{align*}
\EE[x_{ij}]=0,\quad \EE[x_{ij}^2]=1.
\end{align*} 
Moreover, for technical reason, we assume that $x_{ij}$'s satisfy logarithmic Sobolev inequality with constant $\gamma$.
\end{enumerate}

The perturbation $P_N$ is the same as defined in Section \ref{sumdef}. Moreover we require that all the eigenvalues of $P_N$ are larger than $-1$, so that $I_N+P_N$ is positive definite.

\subsection{Orthogonally (or Unitarily) Independent Case}
In this section we study the eigenvalues and eigenvectors of 
\begin{align*}
\tilde{W}_N=(I_N+P_N)^{\frac{1}{2}}U^*H_NU(I_N+P_N)^{\frac{1}{2}}.
\end{align*}
Without loss of generality we can assume that $H_N$ is a diagonal matrix. Since we concentrate only on the  extreme eigenvalues of the perturbed matrix $\tilde{W}_N$, in the following we fix a small universal constant $\delta>0$, and study the eigenvalues of $\tilde{W}_N$  on the spectral domain $(\lambda_N(H_N)-\delta, \lambda_1(H_N)+\delta)^\complement$.
\begin{theorem}{\label{eigfreemult}}
$H_N$ and $P_N$ are defined as in Section \ref{multdef}, we assume that their norms are bounded by some universal constant $B$, i.e. $\|H_N\|\leq B$, $\|P_N\|\leq B$. We denote the eigenvalues of $\tilde{W}_N=(I_N+P_N)^{\frac{1}{2}}U^*H_NU(I_N+P_N)^{\frac{1}{2}}$ by $\lambda_{1}(\tilde{W}_N)\geq \lambda_{2}(\tilde{W}_N)\geq \cdots\geq \lambda_{N}(\tilde{W}_N)$. For any $0<\epsilon\leq \delta$ (it may depend on $M$ and $N$), and any $1\leq i\leq M$, if the eigenvalue $\theta_i$ of the perturbation $P_N$ satisfies the following separation condition:
\begin{align}
\label{freemultsep}
\begin{split}
T_N^{-1}(\frac{1}{\theta_i})\geq \lambda_1(H_N)+2\delta,\quad \text{if } \theta_i>0,\\
T_N^{-1}(\frac{1}{\theta_{i}})\leq \lambda_N(H_N)-2\delta,\quad \text{if } \theta_i<0,
\end{split}
\end{align}
then it creates an extreme eigenvalue $\tilde{\lambda}_i$ of $\tilde{W}_N$, where $\tilde{\lambda}_i=\lambda_i(\tilde{W}_N)$ for $\theta_i>0$, and $\tilde{\lambda}_i=\lambda_{N-M+i}(\tilde{W}_N)$ for $\theta_i<0$, with high probability, 
\begin{align}\label{loc1}
\PP\left(T_N^{-1}(\frac{1}{\theta_i})-\epsilon \leq \tilde{\lambda}_i\leq T_N^{-1}(\frac{1}{\theta_i})+\epsilon\right)\geq 1-C e^{-c_1N\epsilon^2+c_2M}
\end{align}
where $c_1$, $c_2$ and $C$ depend only on $\delta$ and $B$. 
\end{theorem}

\begin{proof}
The proof is similar to the additive perturbation case. For any $z\in (\lambda_N(H_N)-\delta, \lambda_1(H_N)+\delta)^\complement$, $z-H_N$ is invertible, we have
\begin{align*}
\det(z-(I_N+P_N)^{\frac{1}{2}}U^*H_NU(I_N+P_N)^{\frac{1}{2}})
=&\det(z-H_N)\det (I_N-(z-H_N)^{-1}H_N U P_N U^*)\\
=&\det(z-H_N)\det (I_M-\tilde{U}^{*}(z-H_N)^{-1} H_N\tilde{U} \tilde{P}_M),
\end{align*}
where $\tilde{U}$ is the $N\times M$ matrix, consisting of the first $M$ columns of $U$. The extreme eigenvalues of $(I_N+P_N)^{\frac{1}{2}}U^*H_NU(I_N+P_N)^{\frac{1}{2}}$ are solutions of the equation,
\begin{align}
\label{outlierequation}\det (\tilde{P}_M^{-1}-\tilde{U}^{*}(z-H_N)^{-1} H_N\tilde{U})=0.
\end{align}
Denote $D(z)=\tilde{P}_M^{-1}-\tilde{U}^{*}(z-H_N)^{-1} H_N\tilde{U}$. Since $H_N$ is positive semi-definite, one can check that $D(z)$ is non-decreasing on the intervals $(-\infty, \lambda_N(H_N)-\delta]$ and $[\lambda_1(H_N)+\delta, +\infty)$. This implies that the eigenvalues of $D(z)$ are non-decreasing. Since (\ref{outlierequation}) has finitely many solutions, eigenvalues of $D(z)$ can not stay at zero. (In fact one can show that eigenvalues of $D(z)$, as a function of $z$, are either nonzero constant or strictly increasing.) Therefore we can go through the same proof as in Theorem \ref{eigfreesum}. The only difference is: $\tilde{U}^{*}(z-H_N)^{-1} H_N\tilde{U}\approx T_N(z)I_M$.
This is the reason the T-transform appears in the statement instead of the Stieltjes transform.
\end{proof}

\begin{theorem}{\label{evfreemulti}}
$H_N$ and $P_N$ are defined as in Section \ref{multdef}, we assume that their norms are bounded by some universal constant $B$, i.e. $\|H_N\|\leq B$, $\|P_N\|\leq B$. We denote the eigenvalues of $\tilde{W}_N=(I_N+P_N)^{\frac{1}{2}}U^*H_NU(I_N+P_N)^{\frac{1}{2}}$ by $\lambda_{1}(\tilde{W}_N)\geq \lambda_{2}(\tilde{W}_N)\geq \cdots\geq \lambda_{N}(\tilde{W}_N)$. For any $1\leq i\leq M$, such that the eigenvalue $\theta_i$ of the perturbation $P_N$ satisfies the separation condition (\ref{freemultsep}), then with high probability it will create an extreme eigenvalue $\tilde{\lambda}_i$, where $\tilde{\lambda}_i=\lambda_i(\tilde{W}_N)$ for $\theta_i>0$, and $\tilde{\lambda}_i=\lambda_{N-M+i}(\tilde{W}_N)$ for $\theta_i<0$. We denote the corresponding eigenvector by $v_i$ (if $\tilde{\lambda}_i$ is not a single eigenvalue, $v_i$ can be any of its eigenvectors). We denote the projection of $v_i$ on the eigenspace of the nonzero eigenvalues of $P_N$ by $\tilde{v}_i$, which is the projection of $v_i$ on the first $M$ coordinates. For any $0< \epsilon\leq\delta$, it may depend on $M$ and $N$, with probability at least $1-C e^{-c_1N\epsilon^2+c_2M}$, where $c_1$, $c_2$ and $C$ depend only on $\delta$ and $B$, we have
\begin{enumerate}
\item The norm of $\tilde{v}_i$ satisfies
\begin{align}
\label{proj2}\left|\left\|\tilde{v}_i\right\|^2+\frac{\theta_i+1}{\theta_i^2T_N^{-1}(\frac{1}{\theta_i})T_N'(T_N^{-1}(\frac{1}{\theta_i}))}\right|\leq \epsilon C_{B,\delta}.
\end{align}
\item The projection $\tilde{v}_i$ is an approximate eigenvector of $\tilde{P}_M$ with eigenvalue $\theta_i$,
\begin{align}
\label{loc2}\left\|\tilde{P}_M \tilde{v}_i-\theta_i\tilde{v}_i \right\|\leq \epsilon C_{B,\delta},
\end{align}
\end{enumerate}
where $C_{B,\delta}$ is a constant depending on $\delta$ and $B$.
\end{theorem}

\begin{proof}
Since $I_N+P_N$ is positive definite, we can define
\begin{align}
\label{change}w_i=\frac{(I_N+P_N)^{-\frac{1}{2}}v_i}{\|(I_N+P_N)^{-\frac{1}{2}}v_i\|},
\end{align}
which is a normalized eigenvector of $U^*H_NU(I_N+P_N)$. We denote the projection of $w_i$ on the eigenspace of the nonzero eigenvalues of $P_N$ by $\tilde{w}_i$. We have the following two equations for $\tilde{w}_i$,
\begin{align*}
&\left(I_M- \tilde{U}^*(\tilde{\lambda}_i-H_N)^{-1}H_N\tilde{U}\tilde{P}_M\right)\tilde{w}_i=0,\\
&\tilde{w}_i^*\tilde{P}_M \left(\tilde{U}^*H_N(\tilde{\lambda}_i-H_N)^{-2}H_N\tilde{U}\right)\tilde{P}_M\tilde{w}_i=1.
\end{align*}
Approximately, by Proposition \ref{orthogonalcase}, we have $\tilde{U}^*(\tilde{\lambda}_i-H_N)^{-1}H_N\tilde{U}\approx \frac{1}{\theta_i}I_M$ and $\tilde{U}^*H_N(\tilde{\lambda}_i-H_N)^{-2}H_N\tilde{U}\approx \left(-\frac{1}{\theta_i}-T_N^{-1}(\frac{1}{\theta_i})T_N'(T_N^{-1}(\frac{1}{\theta_i}))\right)I_M$. By the same argument as in Theorem \ref{evfreesum}, one can show that $\tilde{w}_i$ is an approximate eigenvector of $\tilde{P}_M$ with eigenvalue $\theta_i$, and its norm is given by \begin{align}
\label{projw2}\left|\left\|\tilde{w}_i\right\|^2+\frac{1}{\theta_i+\theta_i^2T_N^{-1}(\frac{1}{\theta_i})T_N'(T_N^{-1}(\frac{1}{\theta_i}))}\right|\leq \epsilon C_{B,\delta}.
\end{align}
From (\ref{change}), we get $\tilde{v}_i=\|(I_N+P_N)^{-\frac{1}{2}}v_i\|(I_M+\tilde{P}_M)^{\frac{1}{2}}\tilde{w}_i$. So $\tilde{v}_i$ is also an approximate eigenvector of $\tilde{P}_M$ with eigenvalue $\theta_i$. And its norm (\ref{proj2}) can be computed from (\ref{projw2}).
\end{proof}

\subsection{Wishart Case}
In this section we study the eigenvalues of the perturbed Wishart matrices,
\begin{align*}
\tilde{W}_N=(I_N+P_N)^{\frac{1}{2}}W_N(I_N+P_N)^{\frac{1}{2}}.
\end{align*}
Without lose of generality, we assume that $P_N=V^* \tilde{P}_M V$, where $V$ is an $N\times M$ matrix, consisting of the eigenvectors corresponding to the nonzero eigenvalues of $P_N$. It is well known that the empirical eigenvalue distribution of $W_N$ converges to the Marchenko-Pastur law with density
\begin{align*}
\rho_{mp}(x) =\frac{\sqrt{(\gamma_+-x)(x-\gamma_-)}}{2\pi \phi x}, \quad \text{for } x\in [\gamma_-, \gamma_+],
\end{align*}
where $\gamma_-=(1-\sqrt{\phi})^2$ and $\gamma_+=(1+\sqrt{\phi})^2$. We denote its T-transform as
\begin{align*}
T_{mp}(z)=\frac{z-\phi-1 -\text{sgn}(z-\gamma_+)\sqrt{(z-\gamma_-)(z-\gamma_+)}}{2\phi},\quad \text{for } z\in (-\infty, \gamma_-)\cup (\gamma_+, +\infty).
\end{align*}
For any fixed small universal constant $\delta>0$, it is known that with high probability the eigenvalues of the Wishart matrix $W_N$ are in $\left(\gamma_- -\frac{\delta}{2}, \gamma_++\frac{\delta}{2}\right)$. We study the extreme eigenvalues of $\tilde{W}_N$ on the spectral domain $\left(\gamma_--\delta, \gamma_++\delta\right)^\complement$.
\begin{theorem}{\label{wishmul}}
Wishart matrix $W_N=X_NX_N^*$ and perturbation $P_N$ are defined as in Section \ref{multdef}, we assume that the norm of $P_N$ is bounded by some universal constant $B$, i.e. $\|P_N\|\leq B$. We denote the eigenvalues of $\tilde{W}_N=(I_N+P_N)^{\frac{1}{2}}W_N(I_N+P_N)^{\frac{1}{2}}$ by $\lambda_{1}(\tilde{W}_N)\geq \lambda_{2}(\tilde{W}_N)\geq \cdots\geq \lambda_{N}(\tilde{W}_N)$. For any $\sqrt{\frac{M}{N}}\ll\epsilon\leq \delta$ (it may depends on $M$ and $N$), and any $1\leq i\leq M$, if $\theta_i$ satisfies the following separation condition
\begin{align}
\label{wishsep}|\theta_i|\geq \sqrt{\phi}+2\delta
\end{align} 
then it will create an extreme eigenvalue $\tilde{\lambda}_i$ of $\tilde{W}_N$, where $\tilde{\lambda}_i=\lambda_i(\tilde{W}_N)$ for $\theta_i>0$, and $\tilde{\lambda}_i=\lambda_{N-M+i}(\tilde{W}_N)$ for $\theta_i<0$, with high probability, 
\begin{align}\label{wishloc}
\PP\left(\phi+1+\theta_i+\frac{\phi}{\theta_i}-\epsilon\right. \leq \tilde{\lambda}_i\leq\left. \phi+1+\theta_i+\frac{\phi}{\theta_i}+\epsilon\right)
\geq1-Ce^{-c_1N\epsilon^2+c_2M}\end{align}
where $c_1$, $c_2$ and $C$ depend only on $\delta$ and $B$. 
\end{theorem}

\begin{proof}
We denote the set of matrices $D=\{H\in \mbf{M}_{N\times p}: \gamma_{-}-\frac{\delta}{2}\leq \min_{1\leq i\leq N} \lambda_i(HH^*)\leq \max_{1\leq i\leq N}\lambda_i(HH^*)\leq \gamma_{+}+\frac{\delta}{2}\}$. Under our assumption that the entries satisfy the logarithmic Sobolev inequality, it is well known that $\EE[\min_{1\leq i\leq N} \lambda_i(W_N)]=\gamma_{-}$. Therefore by the logarithmic Sobolev inequality we have $\PP(\min_{1\leq i\leq N} \lambda_i(W_N)\leq \gamma_{-}-\frac{\delta}{2})\leq e^{-cN\delta^2}$. Similarly $\PP(\max_{1\leq i\leq N} \lambda_i(W_N)\geq \gamma_{+}+\frac{\delta}{2})\leq e^{-cN\delta^2}$. Therefore we have $\PP(X_N\in D)\geq 1-2e^{-cN\delta^2}$. Since on the region $D$, for any $z\in (\gamma_{-}-\delta, \gamma_{+}+\delta)^{\complement}$ and unit vector $u$, $X\rightarrow u^*(z-XX^*)^{-1}XX^*u$ is Lipschitz continuous with Lipschitz constant at most $O(\delta^{-2})$. Moreover, we also have the isotropic local law for the sample covariance matrices \cite[Theorem 2.4]{BEKYY}. Therefore analogue to Proposition \ref{wignerconcen}, we can prove the following concentration of measure inequality,
\begin{align}{\label{samplecase}}
\begin{split}
\PP(X_N\in D \text{ and }\left\|U^*(z-X_NX_N^*)^{-1}X_NX_N^*U- T_{mp}(z)I_M\right\|\leq& \frac{\xi}{3})\\
\geq& 1- C e^{-cN\xi^2+M\ln 7}, 
\end{split}
\end{align} 
where the constants $c$ and $C$ depend only on $\delta$ and the logarithmic Sobolev constant $\gamma$. With \eqref{samplecase}, the similar argument as in Theorem \ref{eigfreemult} leads to \eqref{wishloc}.
\end{proof}

Similarly as the Wigner case, we have the following corollary on the empirical distribution of extreme eigenvalues of the perturbed Wishart matrices. 
\begin{corollary}
Wishart ensemble $W_N$ and perturbation $P_N$ are defined as in Section \ref{multdef}, we assume the norm of $P_N$ is bounded by some universal constant $B$, i.e. $\|P_N\|\leq B$. Moreover we assume that the empirical distribution of the nonzero eigenvalues of $P_N$ converges weakly to a compactly supported measure $\nu$,
\begin{align*}
\hat{\nu}_{N}=\frac{1}{M}\sum_{i=1}^{M}\delta_{\theta_i}\rightarrow \nu,
\end{align*} 
with support supp $\mu \subset [a, b]$, and $a>\sqrt{\phi}$. We denote the eigenvalues of $\tilde{W}_N=(I_N+P_N)^{\frac{1}{2}}W_N(I_N+P_N)^{\frac{1}{2}}$ by $\lambda_1(\tilde{W}_N)\geq \lambda_2(\tilde{W}_N)\geq \cdots \geq \lambda_N(\tilde{W}_N)$. Almost surely, the empirical distribution of the largest $M$ eigenvalues of $\tilde{W}_N$ converges weakly to the push forward measure $\gamma_\# \nu$, 
\begin{align*}
\frac{1}{M}\sum_{i=1}^{M}\sigma_{\lambda_i(\tilde{W}_N)}\rightarrow \gamma_\#\nu, \quad a.s.
\end{align*}
where $\gamma(\theta)=\phi+1+\theta+\frac{\phi}{\theta}$.
\end{corollary}

\begin{appendix}

\section{Concentration of Measure}
\begin{proposition}{\label{orthogonalcase}}
Let $A$ be a deterministic $N\times N$ matrix. $\tilde{U}$ is the first $M$ columns of an $N\times N$ orthogonal (or unitary) matrix following Haar measure. For any $\xi>0$, we have 
\begin{align}{\label{normbound1}}
\PP\left(\left\|\tilde{U}^*A\tilde{U}-\frac{\Tr A}{N}I_M\right\|\leq \xi\right)\geq 1-Ce^{-cN\xi^2+M\ln 7},
\end{align}
where $c$ and $C$ depend on the norm of $A$.
\end{proposition}
\begin{proof}
We will prove the orthogonal case, the unitary case is exactly the same. Denote $\tilde{A}=A-\frac{\Tr A}{N}I_N$, then $\Tr \tilde{A}=0$. (\ref{normbound1}) is equivalent to show that $\|\tilde{U}^{*}\tilde{A}\tilde{U}\|\leq \xi$ with high probability. 

The proof follows from the standard epsilon net argument. We refer to \cite[Chapter 2.3.1]{Ta} and \cite[Chapter 5.2.2]{Ve} for a detailed discussion on epsilon net argument.  For any $v\in S^{M-1}=\{x\in \RR^M: |x|=1\}$, $\tilde{U}v$ is uniformly distributed on the sphere $S^{N-1}$. From the application of a well-known concentration of measure result, we have
\begin{align}{\label{orthogonalconcentration}}
\PP\left(\left|v^*\tilde{U}^* \tilde{A}\tilde{U}v\right|\geq \frac{\xi}{3}\right)=\PP\left(\left|v^*\tilde{U}^* A\tilde{U}v-\EE\left[v^*\tilde{U}^* A\tilde{U}v\right]\right|\geq \frac{\xi}{3}\right)\leq C e^{-cN\xi^2},
\end{align} 
for any $\xi>0$, where $c$ and $C$ depend on the norm of $A$.

Take $\Sigma$ to be a maximal $\frac{1}{3}$-net of the sphere $S^{M-1}$, i.e. a set of points in $S^{M-1}$ that are separated from each other by a distance of at least $\frac{1}{3}$, and which is maximal with respect to set inclusion. The volume argument gives us that $\Sigma$ has cardinality
\begin{align*}
|\Sigma|\leq (1+\frac{1}{\frac{1}{2}\times\frac{1}{3}})^{M}=7^{M}.
\end{align*}
For any $v\in S^{M-1}$, there exists some $w\in \Sigma$ such that $|v-w|\leq \frac{1}{3}$. Then we have,
\begin{align*}
\left|v^*\tilde{U}^*\tilde{A}\tilde{U}v\right|
\leq& \left|w^*\tilde{U}^*\tilde{A}\tilde{U}w\right|+\left|(v-w)^*\tilde{U}^*\tilde{A}\tilde{U}w\right|+\left|v^*\tilde{U}^*\tilde{A}\tilde{U}(v-w)\right|\\
\leq& \left|w^*\tilde{U}^*\tilde{A}\tilde{U}w\right|+\frac{2}{3}\left\|\tilde{U}^*\tilde{A}\tilde{U}\right\|.
\end{align*}
Therefore, we can replace the sphere $S^{M-1}$ by its $\frac{1}{3}-net$, and get
\begin{align}{\label{sizebound1}}
\left\|\tilde{U}^*\tilde{A}\tilde{U}\right\|=\sup_{v\in S^M}\left|v^*\tilde{U}^*\tilde{A}\tilde{U}v\right|
\leq 3\sup_{v\in\Sigma}\left|v^*\tilde{U}^*\tilde{A}\tilde{U}v\right|.
\end{align}
Combining (\ref{orthogonalconcentration}) and (\ref{sizebound1}), we have
\begin{align*}
\PP(\left\|v^*\tilde{U}^* \tilde{A}\tilde{U}v\right\|\geq \xi)
\leq& \PP\left(\bigcup_{v\in \Sigma}\left\{\left|v^*\tilde{U}^* \tilde{A}\tilde{U}v\right|\geq \frac{\xi}
{3}\right\}\right)\\
\leq& \sum_{v\in\Sigma}\max_{v\in \Sigma}
\left\{\PP\left(\left|v^*\tilde{U}^* \tilde{A}\tilde{U}v\right|\geq \frac{\xi}
{3}\right)\right\}\\
\leq&C e^{-cN\xi^2+M\ln 7}. 
\end{align*}
This finishes the proof.
\end{proof}

\section{Bounds on Stieltjes Transform and T-Transform}
\begin{proposition}{\label{inequalitySandT}}
$H_N$ and $P_N$ are defined as in Section (\ref{sumdef}), such that their norms are bounded by $B$. $m_N$ is the Stieltjes transform of the empirical measure of $H_N$. For any $|\xi| \leq \delta$, we have 
\begin{align}
\label{fsbound} 
&\frac{|\xi|}{4B^2}\leq \left|m_N \left(m_N^{-1}(\frac{1}{\theta_i})+\xi\right)-\frac{1}{\theta_i}\right| \leq \frac{4|\xi|}{ \delta^2}\\
\label{ssbound}&\frac{|\xi|}{8B^3}\leq \left|m_N'\left(m_N^{-1}(\frac{1}{\theta_i})+\xi\right)-m_N'\left( m_N^{-1}(\frac{1}{\theta_i})\right)\right| \leq \frac{8|\xi|}{ \delta^3}
\end{align}
 where $\theta_i$ is an eigenvalue of $P_N$ such that 
$m_N^{-1}(\frac{1}{\theta_i})\geq \lambda_1(H_N)+2\delta$, or  $m_N^{-1}(\frac{1}{\theta_i})\leq \lambda_N(H_N)-2\delta$. 

$H_N$ and $P_N$ are defined as in Section (\ref{multdef}), such that their norms are bounded by $B$. $T_N$ is the Stieltjes transform of the empirical measure of $H_N$. For any $|\xi| \leq \delta$, we have
\begin{align}
\label{ftbound} 
&\frac{|\xi|}{4B^3}\leq \left|T_N \left(T_N^{-1}(\frac{1}{\theta_i})+\xi\right)-\frac{1}{\theta_i}\right| \leq \frac{4B|\xi|}{ \delta^3}\\
\label{stbound}&\frac{|\xi|}{8B^4}\leq \left|T_N'\left(T_N^{-1}(\frac{1}{\theta_i})+\xi\right)-T_N'\left( T_N^{-1}(\frac{1}{\theta_i})\right)\right| \leq \frac{8B|\xi|}{ \delta^4}
\end{align}
where $\theta_i$ is an eigenvalue of $P_N$ such that $T_N^{-1}(\frac{1}{\theta_i})\geq \lambda_1(H_N)+2\delta$, or $T_N^{-1}(\frac{1}{\theta_i})\leq \lambda_N(H_N)-2\delta$. 
\end{proposition}
\begin{proof}
We will only prove (\ref{fsbound}) for positive eigenvalues of $P_N$, the other inequalities can be proved in the same way. On the interval $(\lambda_1(H_N)+\delta, +\infty)$, $m_N(z)$ is positive and strictly decreasing. By Taylor expansion, there exists some $\gamma\in [0,1]$ such that
\begin{align}
\notag \left|m_N(m_N^{-1}(\frac{1}{\theta_i})+\xi)- \frac{1}{\theta_i}\right|
=&\left|\xi m_N'(m_N^{-1}(\frac{1}{\theta_i})+\gamma\xi)\right|\\
\notag=&\frac{|\xi|}{N}\sum_{k=1}^{N}\frac{1}{\left(m_N^{-1}(\frac{1}{\theta_i})+\gamma\xi-\lambda_k(H_N)\right)^2}\\
\label{removeep}\geq&\frac{|\xi|}{4N}\sum_{k=1}^{N}\frac{1}{\left(m_N^{-1}(\frac{1}{\theta_i})-\lambda_k(H_N)\right)^2}\\
\label{cauchy}\geq&|\xi| \left(\frac{1}{2N}\sum_{k=1}^{N}\frac{1}{m_N^{-1}(\frac{1}{\theta_i})-\lambda_k(H_N)}\right)^2\\
\notag=&\frac{|\xi|}{(4\theta_i)^2}\geq\frac{|\xi|}{4B^2}
\end{align}
For the inequality (\ref{removeep}), we use the fact that $m_N^{-1}(\frac{1}{\theta_i})-\lambda_k(H_N)\geq2\delta \geq \delta+ |\xi|$. (\ref{cauchy}) is from Jensen's inequality. For the upper bound,
\begin{align*}
 \left|m_N(m_N^{-1}(\frac{1}{\theta_i})+\xi)- \frac{1}{\theta_i}\right|
\leq\frac{4|\xi|}{N}\sum_{k=1}^{N}\frac{1}{\left(m_N^{-1}(\frac{1}{\theta_i})-\lambda_k(H_N)\right)^2}
\leq\frac{4|\xi|}{\delta^2}.
\end{align*}
\end{proof}
\end{appendix}

\section*{Acknowledgement}
The author expresses grateful thanks to Prof. Alice Guionnet, for introducing me to this problem, helpful discussions throughout the research process and constructive suggestions on the first draft of this paper.

\printbibliography
\end{document}